\documentclass{lms}
\usepackage{amsmath}
\usepackage{amssymb}
\usepackage{enumerate}
\usepackage{algorithmic}
\usepackage{algorithm}
\usepackage{tikz}
\usetikzlibrary{trees}

\newtheorem{theorem}{Theorem}[section] 
\newtheorem{lemma}[theorem]{Lemma}     

\newtheorem{proposition}[theorem]{Proposition}

\newnumbered{definition}[theorem]{Definition}
\newnumbered{remark}[theorem]{Remark}
\newnumbered{remarks}[theorem]{Remarks}
\newnumbered{heuristic}[theorem]{Heuristic}
\newnumbered{heuristics}[theorem]{Heuristics}
\newnumbered{lmsalgorithm}[theorem]{Algorithm}
\newnumbered{subalgorithm}[theorem]{Subalgorithm}

\newcommand{\url}[1]{\texttt{#1}}
\newcommand{\nrd}{\mathrm{nrd}}
\newcommand{\GL}{\mathrm{GL}}
\newcommand{\SL}{\mathrm{SL}}
\newcommand{\PGL}{\mathrm{PGL}}
\newcommand{\PSL}{\mathrm{PSL}}
\newcommand{\Cl}{\mathrm{Cl}}

\newcommand{\fld}{K}
\newcommand{\fldfin}[1]{\kappa_{#1}}
\newcommand{\fldext}{L}
\newcommand{\fldb}{F}
\newcommand{\alg}{A}
\newcommand{\secompl}{L(\Delta)^{O(1)}}
\newcommand{\clalg}{\Cl_{\alg}(\fld)}
\newcommand{\fb}{\mathcal{B}}
\newcommand{\order}{\mathcal{O}}
\newcommand{\orderb}{\mathcal{O}}
\newcommand{\ring}{\mathcal{R}}
\newcommand{\mat}{\mathcal{M}}
\newcommand{\tree}{\mathcal{T}}
\newcommand{\Z}{\mathbb{Z}}
\newcommand{\Proj}{\mathbb{P}}
\newcommand{\Q}{\mathbb{Q}}
\newcommand{\R}{\mathbb{R}}
\newcommand{\C}{\mathbb{C}}
\newcommand{\Hamil}{\mathbb{H}}

\newcommand{\idl}{\mathfrak{a}}

\newcommand{\idlb}{\mathfrak{b}}
\newcommand{\tsidl}{\mathfrak{I}}
\newcommand{\cnd}{\mathfrak{f}}
\newcommand{\prm}{\mathfrak{p}}
\newcommand{\Prm}{\mathfrak{P}}
\newcommand{\act}{\cdot}
\newcommand{\sto}{\leftarrow}
\newcommand{\floor}[1]{\lfloor #1 \rfloor}
\newcommand{\gp}[1]{\langle #1 \rangle}
\newcommand{\disc}{{d}}
\newcommand{\na}{\mathcal{N}}

\newcommand{\markscale}{1.7}
\newcommand{\scsh}{2.5cm}
\tikzset{btt/.style={%
  thick,
  scale=1.2,
  grow cyclic,
  level distance=1.cm,
  level 1/.style={sibling angle=120},
  level 2/.style={sibling angle=90},
  level 3/.style={sibling angle=90},
  level 4/.style={sibling angle=55},
  nodes={circle,draw,inner sep=+0pt, minimum size=4pt}
}}

\newcommand{\levelc}{%
  child[densely dotted] foreach \cntV in {1,2} {}
}

\newcommand{\levelq}{%
  child[solid] foreach \cntIV in {1,2} {
    node[solid]{}
    \levelc
  }
}

\newcommand{\levelt}{%
  child[solid] foreach \cntIII in {1,2} {
    node[solid]{{}}
    \levelq
  }
}

\newcommand{\leveld}{%
  child[solid] foreach \cntII in {1,2} { 
    node[solid] {}
    \levelt
  }
}

\newcommand{\levelu}{%
  \path[rotate=30]
  node[solid] {}
  child[solid] foreach \cntI in {1,2,3} {
    node[solid] {}
    \leveld
  };
}

\newcommand{\INPUT}{\REQUIRE}

\newcommand{\OUTPUT}{\ENSURE}

\newcommand{\LINEIF}[2]{%
    \STATE\algorithmicif\ {#1}\ \algorithmicthen\ {#2} \algorithmicend\ \algorithmicif%
}

\title[The principal ideal problem in quaternion algebras]
 {An algorithm for the principal ideal problem\\in indefinite quaternion algebras} 

\author{A. Page}

\classno{11Y40 (primary), 11R52, 11R65 (secondary)}

\extraline{This research was partially funded by ERC Starting Grant ANTICS 278537.}

\begin{document}
\maketitle


\begin{abstract}
  Deciding whether an ideal of a number field is principal and finding a generator is a fundamental problem with many applications in computational number theory. For indefinite quaternion algebras, the decision problem reduces to that in the underlying number field. Finding a generator is hard, and we present a heuristically subexponential algorithm.
\end{abstract}

\textbf{Keywords}: quaternion algebra, principal ideal algorithm, factor base, Bruhat-Tits tree.

\section{Introduction}
Automorphic forms and their Hecke eigenvalues are of tremendous importance in number theory. These eigenvalues carry a lot of interesting arithmetic information, such as the number of points on elliptic curves or traces of Frobenius in Galois representations. One of the most successful methods for computing automorphic forms for~$\GL_2$ over number fields uses the Jacquet-Langlands correspondence.
This result transfers the problem to a quaternion algebra, in which it is often easier to solve. This approach has its roots in the theory of Brandt matrices and has been successfully used by Demb\'el\'e-Donnelly 
and Greenberg-Voight~\cite{vd} 
to compute Hecke eigenvalues of Hilbert modular forms.
In both methods, a crucial step is to test whether an ideal is principal and to produce a generator in this case: this is the \emph{principal ideal problem} that we are considering in this paper.
\par The principal ideal problem naturally splits into two cases: definite and indefinite algebras. In the definite case, Demb\'el\'e and Donnelly described an algorithm and Kirschmer and Voight proved that this algorithm runs in polynomial time when the base field is \emph{fixed}, so we focus on the remaining indefinite case. In that case, testing whether an ideal is principal reduces to the same problem over the base field by Eichler's theorem (Theorem~\ref{thmeichlerclass}), but finding a generator is difficult. Kirschmer and Voight~\cite{kv} provide an algorithm that improves on naive enumeration\footnote{Trying every linear combination of a basis until we find a generator.}, without analysing its complexity.
\par In this paper, we present a probabilistic algorithm using a factor base and an auxiliary data structure to solve the principal ideal problem. Our algorithm is inspired by Buchmann's algorithm~\cite{buchmann} for computing the class group of a number field. However, it is not easy to adapt this technique to quaternion algebras. Indeed, the set of right ideals of an order does not form a group under multiplication. %
In fact, for most pairs of ideals, multiplication is not well-defined. We are able to salvage the factor base technique in the case of indefinite quaternion algebras by algorithmically realizing the strong approximation property (Theorem~\ref{thmstrongapprox}). %
The main point is that if every ideal were two-sided, Buchmann's method would work unchanged. Our algorithm is divided in two parts. Because the algebra is indefinite, every ideal is equivalent to an ``almost two-sided'' ideal: a local algorithm (Algorithm~\ref{preduce}) makes this equivalence effective. The global algorithm (Algorithm~\ref{greduce}) uses a factor base: by linear algebra it cancels out the valuations of the \emph{norm} of the ideal and then corrects the ideal locally at every prime to make it two-sided.
We implemented our algorithm as in Magma. It performs well in practice, compared to the built-in Magma function implementing Kirschmer and Voight's algorithm.

\par The paper is organized as follows. We first recall basic properties of quaternion algebras, Eichler's theorems and Bruhat-Tits trees in Section~\ref{secrappels}. We then proceed to algorithms in Section~\ref{secalgos}. In Section~\ref{secred}, we define local and global reduction structures and Algorithm~\ref{isprincipal}, solving the principal ideal problem. In Section~\ref{secstruct}, Algorithm~\ref{gbuild} constructs the needed local and global reduction structures: %
the first one uses units constructed from commutative suborders, and the second one is inspired by Buchmann's algorithm. In Section~\ref{seccr}, we introduce a compact representation for quaternions to prevent coefficient explosion in the previous algorithms. Section~\ref{seccomplexity} provides a complexity analysis of our algorithms: assuming suitable heuristics, we prove a subexponential running time. Section~\ref{secex} presents examples.%

\section{Background on quaternion algebras and Bruhat-Tits trees}\label{secrappels}

When~$G$ is a group and~$S\subset G$ is a subset, we write~$\gp{S}$ for the subgroup generated by~$S$. When the group~$G$ acts on a set~$X$, we say that~$S$ acts transitively on~$X$ if~$\gp{S}$ does.

\subsection{Quaternion algebras}
Good references for this section and the next one are~\cite{kv},~\cite{mfv} and~\cite{voightbook}. Let~$\fld$ be a number field with ring of integers~$\Z_\fld$ and discriminant~$\disc_\fld$. We write~$N:\fld\to\Q$ for the norm. 
Let~$\prm$ be a prime of~$\Z_\fld$. We write~$\fld_\prm$ for the~$\prm$-adic completion of~$\fld$, we let $v_\prm$ be the $\prm$-adic valuation and we write~$\fldfin{\prm}$ for the residue field~$\Z_\fld/\prm$. When~$S$ is a set of primes of~$\Z_\fld$, we write~$\Z_{\fld,S}$ the ring of~$S$-integers in~$\fld$.
\par Let~$\alg$ be a quaternion algebra over~$\fld$ with reduced norm~$\nrd$. Let~$v$ be a place of~$\fld$. The place~$v$ is \emph{split} or~\emph{ramified} according to whether~$\alg\otimes_{\fld}\fld_v\cong\mat_2(\fld_v)$ or not. The \emph{reduced discriminant~$\delta_\alg$} of~$\alg$ is the product of the ramified primes and its \emph{absolute discriminant} is the integer~$\Delta_\alg=\disc_\fld^4N(\delta_\alg)^2$.
Let~$\order$ be a maximal order in~$\alg$. We write~$\order^1$ for the group~$\{x\in\order\ |\ \nrd(x)=1\}$. A \emph{lattice}~$I\subset\alg$ is a finitely generated~$\Z_\fld$-submodule such that~$\fld I=\alg$. The \emph{right order}~$\order_r(I)$ of~$I$ is the set~$\{x\in\alg\ |\ Ix\subset I\}$ and the \emph{left order~$\order_l(I)$} is defined analogously. A \emph{right $\order$-ideal} is a lattice~$I$ such that~$\order_r(I)=\order$. The ideal~$I$ is \emph{integral} if~$I\subset\order$ and~$I$ is \emph{two-sided} if~$\order_l(I)=\order_r(I)$. The inverse~$I^{-1}$ of~$I$ is~$\{x\in\alg\ |\ xI\subset\order\}$. If~$I,J$ are lattices such that~$\order_r(I)=\order_l(J)$, we define their product~$IJ$ to be the lattice generated by the set~$\{xy\ :\ x\in I,\ y\in J\}$. If~$I$ is an~$\order$-ideal we have~$II^{-1}=\order_l(I)$ and~$I^{-1}I=\order_r(I)$. The \emph{reduced norm~$\nrd(I)$} of an $\order$-ideal~$I$ is the~$\Z_\fld$-module generated by the reduced norms of elements in~$I$. The reduced norm of ideals is 
multiplicative. %
For a right~$\order$-ideal~$I$ we define~$\na(I)=N(\nrd(I))$ and for an element~$x\in\alg^\times$ we set~$\na(x)=\na(x\order)$.
Let~$\prm$ be a prime of~$\Z_\fld$. There exists a unique two-sided~$\order$-ideal~$\Prm$ such that every two-sided $\order$-ideal having reduced norm a power of~$\prm$ is a power of~$\Prm$. We have~$\Prm=\prm\order$ if~$\prm$ splits in~$\alg$ and~$\Prm^2=\prm\order$ if~$\prm$ ramifies in~$\alg$: such an ideal~$\Prm$ is called a prime of~$\order$, and every two-sided $\order$-ideal is a product of primes of~$\order$.
The set of right~$\order$-ideals is equipped with an action of the group of two-sided~$\order$-ideals by multiplication on the right and an action of the group~$\alg^\times$ by multiplication on the left. Two right~$\order$-ideals~$I,J$ are \emph{equivalent} if there exists~$x\in\alg^\times$ such that~$xI=J$, that is if they lie in the same orbit modulo~$\alg^\times$. The set of equivalence classes of right~$\order$-ideals is written~$\Cl(\order)$. An ideal is~\emph{principal} if it is equivalent to the unit ideal~$\order$. If~$S$ is a set of primes of~$\Z_\fld$, the~\emph{$S$-order} associated with~$\order$ is the ring~$\order_S=\Z_{\fld,S}\order$ and the group of~\emph{$S$-units} (relative to~$\order$) in~$\alg$ is~$\order_S^\times$.

\subsection{Eichler's condition and theorems}

A quaternion algebra~$\alg$ satisfies the \emph{Eichler condition} or is~\emph{indefinite} if there exists an infinite place of the base field~$\fld$ at which~$\alg$ is split. Indefinite algebras satisfy the following properties.

\begin{theorem}[(Consequence of strong approximation)]\label{thmstrongapprox}
  Let~$\order$ be a maximal order in a quaternion algebra~$\alg$ over a number field~$\fld$, satisfying the Eichler condition. Let~$\prm$ be a prime of~$\Z_\fld$ that splits in~$\alg$ and~$k$ a positive integer. Then the map
  \[\order^1 \longrightarrow \SL_2(\Z_\fld/\prm^k)\]
  is surjective.
\end{theorem}

\begin{theorem}[(Integral version of Eichler's norm theorem)]\label{thmeichlernorm}
  Let~$\order$ be a maximal order in a quaternion algebra~$\alg$ over a number field~$\fld$ satisfying the Eichler condition. Let~$S$ be a finite set of primes of~$\Z_\fld$. Let~$\Z_{\fld,S,\alg}^\times$ be the set of~$S$-units that are positive at every real place of~$\fld$ that ramifies in~$\alg$. Then the reduced norm
  \[\nrd : \order_S^\times\longrightarrow \Z_{\fld,S,\alg}^\times\]
  is surjective.
\end{theorem}

\begin{theorem}[(Eichler)]\label{thmeichlerclass}
  Let~$\order$ be a maximal order in a quaternion algebra~$\alg$ over a number field~$\fld$ satisfying the Eichler condition. Let~$\clalg$ be the ray class group with modulus the product of the real places of~$\fld$ that ramify in~$\alg$. Then the reduced norm induces a bijection
  \[\Cl(\order)\xrightarrow{\sim}\clalg\text{.}\]
\end{theorem}
In other words, two right~$\order$-ideals are equivalent if and only if the classes of their norm in~$\clalg$ are equal. Note that since~$\Cl(\order)$ is not a group, this map is only a bijection of \emph{sets}.

\subsection{The Bruhat-Tits tree}\label{secbtt}
The standard reference for this section is~\cite{serretrees}. Let~$\fld$ be a field with a discrete valuation~$v$. Let~$R$ be its valuation ring,~$\pi$ a uniformizer and~$\fldfin{}=R/\pi R$ the residue field. An \emph{$R$-lattice in~$\fld^2$} is an~$R$-submodule of rank~$2$ in~$\fld^2$. We define the \emph{Bruhat-Tits tree~$\tree$}, which we write~$\tree_\prm$ when~$\fld$ is the $\prm$-adic completion of a number field. The set of vertices of~$\tree$ is the set of homothety classes of $R$-lattices in~$\fld^2$. Let~$L,L'$ be two such $R$-lattices. There exists an ordered~$R$-basis~$(e_1,e_2)$ of~$L$ and integers~$a,b$ such that~$(\pi^a e_1,\pi^b e_2)$ is an $R$-basis of~$L'$. The integer~$|a-b|$ depends only on the homothety classes of~$L,L'$ and is called their \emph{distance}.
By definition, there is an edge in the tree~$\tree$ between every pair of vertices at distance~$1$. The graph~$\tree$ is an infinite tree. If~$P,Q$ are two vertices, the unique path of minimum length between~$P$ and~$Q$ is called the \emph{segment}~$PQ$ and the distance~$d(P,Q)$ equals the length of the segment~$PQ$. The set of vertices at distance~$1$ from a given vertex is in natural bijection with~$\Proj^1(\fldfin{})$.
The group~$\GL_2(\fld)$ acts on the tree and preserves the distance, and this action factors through~$\PGL_2(\fld)$ and is transitive on the set of vertices. The stabilizer of the vertex~$P_0$ corresponding to the $R$-lattice~$R^2$ is~$\fld^\times\GL_2(R)$ and the stabilizer of any vertex is a conjugate of this group. The group~$\SL_2(R)$ acts transitively on the set of vertices at a fixed distance from~$P_0$. For every~$g\in\mat_2(R)\setminus\pi\mat_2(R)$, the Smith normal form shows that~$d(g\act P_0,P_0)=v(\det(g))$. The tree is illustrated in Figure~\ref{figtree} where we label some vertices~$P$ with a matrix~$g$ such that~$P=g\act P_0$.

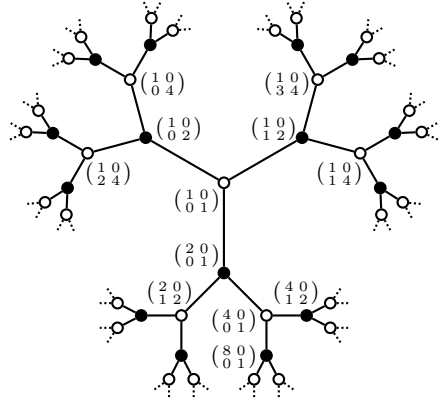
\begin{figure}[thb]
\begin{center}
\begin{tikzpicture}[btt,
  level/.style={
    level distance/.expanded=\ifnum#1>1 \tikzleveldistance/1.5\else\tikzleveldistance\fi,
    nodes/.expanded={\ifodd#1 fill=black\else fill=white\fi}
  }]
  \draw (-0.3,-0.2) node[draw=none,scale=0.9] {$\left(\begin{smallmatrix}1&0\\0&1\end{smallmatrix}\right)$};
  \draw (-0.3,-0.8) node[draw=none,scale=0.9] {$\left(\begin{smallmatrix}2&0\\0&1\end{smallmatrix}\right)$};
  \draw (-0.6,-1.22) node[draw=none,scale=0.9] {$\left(\begin{smallmatrix}2&0\\1&2\end{smallmatrix}\right)$};
  \draw (0.11,-1.53) node[draw=none,scale=0.9] {$\left(\begin{smallmatrix}4&0\\0&1\end{smallmatrix}\right)$};
  \draw (0.11,-1.93) node[draw=none,scale=0.9] {$\left(\begin{smallmatrix}8&0\\0&1\end{smallmatrix}\right)$};
  \draw (0.8,-1.22) node[draw=none,scale=0.9] {$\left(\begin{smallmatrix}4&0\\1&2\end{smallmatrix}\right)$};
  
  \draw (-0.5,0.6) node[draw=none,scale=0.9] {$\left(\begin{smallmatrix}1&0\\0&2\end{smallmatrix}\right)$};
  \draw (-0.68,1.1) node[draw=none,scale=0.9] {$\left(\begin{smallmatrix}1&0\\0&4\end{smallmatrix}\right)$};
  \draw (-1.28,0.1) node[draw=none,scale=0.9] {$\left(\begin{smallmatrix}1&0\\2&4\end{smallmatrix}\right)$};
  
  \draw (0.55,0.6) node[draw=none,scale=0.9] {$\left(\begin{smallmatrix}1&0\\1&2\end{smallmatrix}\right)$};
  \draw (1.28,0.1) node[draw=none,scale=0.9] {$\left(\begin{smallmatrix}1&0\\1&4\end{smallmatrix}\right)$};
  \draw (0.7,1.1) node[draw=none,scale=0.9] {$\left(\begin{smallmatrix}1&0\\3&4\end{smallmatrix}\right)$};
  \levelu
\end{tikzpicture}
\end{center}
\caption{The Bruhat-Tits tree for~$\fld=\Q_2$}
\label{figtree}
\end{figure}

\begin{theorem}\label{thmbtt}
  Let~$P,Q$ be two vertices of the tree~$\tree$ with~$d(P,Q)=1$. Then the action of the group~$G=\SL_2(\fld)$ on the vertices of~$\tree$ has exactly two orbits~$G\act P$ and~$G\act Q$.
\end{theorem}

The connection between the Bruhat-Tits tree and ideals is the following: a right $\mat_2(R)$-ideal is always principal, generated by an element of~$\GL_2(\fld)$. Such an ideal is two-sided if and only if it is generated by an element of~$\fld^\times\GL_2(R)$. So there is a~$\GL_2(\fld)$-equivariant bijection between set of the vertices of the Bruhat-Tits tree and the quotient of the set of right~$\mat_2(R)$-ideals modulo the action of the group of two-sided $\mat_2(R)$-ideals.

\section{Algorithms}\label{secalgos}
We want to adapt the classical subexponential algorithms for computing the class group of a number field due to Hafner and McCurley~\cite{hmc} in the quadratic case and Buchmann~\cite{buchmann} in the general case to indefinite quaternion algebras by using a factor base: a fixed finite set of primes of~$\Z_\fld$. To simplify the notations, we set~$\Delta=\Delta_\alg$.

\begin{definition}\label{deffb}
The \emph{factor base} for~$\alg$ is a finite set~$\fb$ of primes of~$\Z_\fld$ that generates the group~$\clalg$.
\end{definition}

We say that a fractional ideal~$\idl$ of~$\fld$ is~\emph{$\fb$-smooth} or simply \emph{smooth} if it is a product of the primes in~$\fb$. Let~$I$ be a right~$\order$-ideal~$I$. When~$I$ is integral, we say that~$I$ is \emph{smooth} if its reduced norm is. When~$I$ is arbitrary, it is \emph{smooth} if it can be written~$I=J\idl$ with~$\idl$ a smooth fractional ideal of~$\fld$ and $J$ an integral smooth right $\order$-ideal. Equivalently, the ideal~$I$ is smooth if and only if~$I_\prm=\order_\prm$ for all~$\prm \notin \fb$. An element~$x\in A^\times$ is \emph{smooth} if the ideal~$x\order$ is smooth, or equivalently if~$x\in\order_S^\times$ with~$S=\fb$.

\medskip

We equip~$\mat_2(\R)$ and~$\mat_2(\C)$ with the usual positive definite quadratic form~$Q$ given by the sum of the squares of the absolute values of the coefficients, and we equip the Hamiltonian quaternion algebra~$\Hamil$ with the positive definite quadratic form~$Q=\nrd$. For each infinite place of~$\fld$ represented by a complex embedding~$\sigma$, we fix an isomorphism~$\sigma': \alg\otimes_{\fld}\fld_{\sigma}\cong M$ extending~$\sigma$, where~$M$ is one of~$\mat_2(\R)$, $\mat_2(\C)$ or~$\Hamil$. This defines a positive definite quadratic form~$T_2 : \alg\otimes_{\Q}\R \to \R$ by setting~$T_2(x) = \sum_{\sigma}[\fld_{\sigma}:\R]\cdot Q(\sigma'(x))$ for all~$x\in\alg\otimes_{\Q}\R$, giving covolume~$\Delta^{1/2}$ to the lattice~$\order$.
We represent a lattice in~$\alg$ by a $\Z_\fld$-pseudobasis (see~\cite{kv}). When~$L$ is a lattice in~$\alg$, we can enumerate its elements by increasing value of~$T_2$ with the Kannan--Fincke--Pohst algorithm~\cite{fp,kannan}. We represent this enumeration with a routine~\texttt{NextElement} that outputs a new element of~$L$ every time we call~\texttt{NextElement}($L$), ordering them by increasing value of~$T_2$.

\subsection{The reduction algorithms}\label{secred}
In this section, we describe the reduction structures and the corresponding reduction algorithms.
We start with the local reduction, which is an effective version of the fact that every integral right~$\order$-ideal of norm~$\prm^2$ is equivalent to the two-sided ideal~$\prm\order$ (Theorem~\ref{thmeichlerclass}). We perform this reduction by making algorithmic the reduction theory of~$\SL_2(\fld_\prm)$ on the Bruhat-Tits tree~$\tree_\prm$ (Section~\ref{secbtt}). The point is that this reduction needs only a small number of units: this leads to the definition of the $\prm$-reduction structure.

\begin{definition}
  Let~$\prm$ be a prime that splits in~$\alg$, let~$\order_0=\order$ and let~$P_0$ be the fixed point of~$\order_0^\times$ in the Bruhat-Tits tree~$\tree_\prm$. A \emph{$\prm$-reduction structure} is given by the following data:
  \begin{enumerate}
    \item the left order~$\order_1$ of an integral right $\order$-ideal of norm~$\prm$, and the fixed point~$P_1$ of~$\order_{1}^\times$ in~$\tree_\prm$;
    \item for each~$b\in\{0,1\}$ and for each~$P\in\tree_\prm$ at distance~$1$ from~$P_b$, an element~$g\in\order_b^\times$ such that~$g\act P = P_{1-b}$.
  \end{enumerate}
\end{definition}

Such a structure exists by strong approximation (Theorem~\ref{thmstrongapprox}). Note that if~$I$ is an integral right~$\order$-ideal of norm~$\prm$ such that~$\order_1=\order_l(I)$, we have~$I_\prm = x\order_\prm$ for some~$x\in\alg^\times$, and~$P_1 = x\act P_0$. We represent the points at distance~$1$ from~$P_b$ by elements of~$\Proj^1(\fldfin{\prm})$ and we compute the action on these points via explicit splitting maps~$\iota_{b}: \order_b/\prm\order_b \to \mat_2(\fldfin{\prm})$.

\medskip

This structure provides everything we need to perform reduction in the Bruhat-Tits tree. The following algorithm corresponds to the standard reduction procedure (Theorem~\ref{thmbtt}), which is illustrated in Figure~\ref{figred}. The idea is to use successive ``rotations'' (elements in~$\SL_2(\fld_\prm)$ having a fixed point in the tree) around the adjacent vertices~$P_0$ and $P_1$ to send an arbitrary vertex to one of the vertices~$P_b$: every rotation around a vertex decreases the distance to the other one.

\medskip

To realize this procedure, we need to perform the following subtask: given a right $\order$-ideal~$I$, find~$x\in I$ such that~$I_{\prm} = x\order_{\prm}$. A simple idea is to let~$e=v_\prm(\nrd(J))$ and to draw elements~$x\in J/\prm\order$ uniformly at random until~$v_\prm(\nrd(x))=e$. To obtain a deterministic algorithm, we can adapt Euclid's algorithm in the matrix ring~$\mat_2(\ring)$ with~$\ring = \Z_\fld/\prm^{e+1}$. This is done in~\cite{euclmat}, except that the base ring~$\ring$ is assumed to be a domain. We adapt the argument to our case. First note that we have a well-defined~$\prm$-adic valuation~$v=v_\prm$ in the ring~$\ring$. Let~$a,b\in\ring$. We have~$v(ab)\le v(a)+v(b)$ whenever~$ab\neq 0$, and~$a\mid b$ if and only if~$v(b)\ge v(a)$. If~$a\neq 0$, there is a Euclidean division taking the following simple form: if~$a\mid b$ then~$b = a\cdot (b/a) + 0$, and otherwise~$b = a\cdot 0 + b$. In every case we have written~$b = aq+r$ 
with~$r=0$ or~$v(r)<v(a)$.
Adapting this in the matrix ring leads to the following Euclidean division algorithm, where for convenience we write~$w = v\circ\det$. The idea is to work with~$A$ in Smith normal form, and if~$A$ is a diagonal matrix, dividing by~$A$ is almost the same as dividing by the diagonal coefficients. The difference is that we have to ensure that~$\det(R)\neq 0$ unless~$R=0$.

\begin{subalgorithm}[\texttt{DivideMatrix}]\label{divmat}
\begin{algorithmic}[1]
	\INPUT two matrices~$A,B\in\mat_2(\ring)$ with~$\det A \neq 0$, where~$\ring = \Z_F/\prm^i$.
	\OUTPUT two matrices~$Q,R\in\mat_2(\ring)$ such that~$B = AQ+R$, and~($R=0$ or~$w(R)<w(A)$).
	\STATE let~$A'=UAV$ be the Smith form of~$A$ with~$U,V\in\SL_2(\ring)$ and~$A = (\begin{smallmatrix}a & 0\\ 0 & b\end{smallmatrix})$
	\STATE $B'\sto UB$
	\STATE let~$B''=B'W$ be the Hermite form of~$B'$ with~$W\in\SL_2(\ring)$ and~$B'' = (\begin{smallmatrix}c & 0\\ e & f\end{smallmatrix})$
	\IF{$a\mid c$}
	  \IF{$b\mid f$}
	    \IF{$b\mid e$}
	      \STATE $Q \sto \bigl(\begin{smallmatrix}c/a & 0\\ e/b & f/b\end{smallmatrix}\bigr)$, $R \sto \bigl(\begin{smallmatrix}0 & 0\\ 0 & 0\end{smallmatrix}\bigr)$\label{stepcompletediv}
	    \ELSE
	      \STATE $Q \sto \bigl(\begin{smallmatrix}c/a & 1\\ 0 & f/b\end{smallmatrix}\bigr)$, $R \sto \bigl(\begin{smallmatrix}0 & -a\\ e & 0\end{smallmatrix}\bigr)$
	    \ENDIF
	  \ELSE
	    \STATE $Q \sto \bigl(\begin{smallmatrix}c/a-1 & 0\\ 0 & 0\end{smallmatrix}\bigr)$, $R \sto \bigl(\begin{smallmatrix}a & 0\\ e & f\end{smallmatrix}\bigr)$
	  \ENDIF
	\ELSE
	  \IF{$b\mid f$}
	    \STATE $Q \sto \bigl(\begin{smallmatrix}0 & 0\\ 0 & f/b-1\end{smallmatrix}\bigr)$, $R \sto \bigl(\begin{smallmatrix}c & 0\\ e & b\end{smallmatrix}\bigr)$
	  \ELSE
	    \STATE $Q \sto \bigl(\begin{smallmatrix}0 & 0\\ 0 & 0\end{smallmatrix}\bigr)$, $R \sto \bigl(\begin{smallmatrix}c & 0\\ e & b\end{smallmatrix}\bigr)$
	  \ENDIF
	\ENDIF
	\RETURN $VQW^{-1}$, $U^{-1}RW^{-1}$
\end{algorithmic}
\end{subalgorithm}

\begin{proposition}
  Subalgorithm~\ref{divmat} is correct.
\end{proposition}
\begin{proof}
  By case-by-case analysis, we have~$B'' = A'Q+R$, and either~$R=0$ (Step~\ref{stepcompletediv}) or~$\det(R)\neq 0$ and~$w(R)<w(A')$. Let~$Q' =VQW^{-1}$ and~$R'=U^{-1}RW^{-1}$ be the matrices returned by the algorithm. We have~$B = U^{-1}B' = U^{-1}B''W^{-1} = U^{-1}(A'Q+R)W^{-1} = U^{-1}A'V^{-1}Q'+R' = AQ'+R'$. Since~$U,V$ and~$W$ have determinant~$1$, we have~$R=0$ if and only if~$R'=0$, and~$w(R')=w(R)<w(A')=w(A)$, proving the correctness of the algorithm.
\end{proof}

\begin{subalgorithm}[\texttt{GCDMatrix}]\label{gcdmat}
\begin{algorithmic}[1]
  \INPUT two matrices~$A,B\in\mat_2(\ring)$ with~$\det A \neq 0$, where~$\ring = \Z_F/\prm^i$.
  \OUTPUT a matrix~$D$ such that~$A\mat_2(\ring)+B\mat_2(\ring) = D\mat_2(\ring)$.
  \STATE $Q,R\sto $ \texttt{DivideMatrix}($A$, $B$)
  \IF{$R=0$}
    \RETURN $A$
  \ELSE
    \RETURN \texttt{GCDMatrix}($R$, $A$)\label{steprec}
  \ENDIF
\end{algorithmic}
\end{subalgorithm}

\begin{proposition}
  Subalgorithm~\ref{gcdmat} is correct.
\end{proposition}
\begin{proof}
  In Step~\ref{steprec} we have~$\det(R)\neq 0$ by the properties of Subalgorithm~\ref{divmat}, so the recursive call to~\texttt{GCDMatrix} is valid. The rest of the proof is the same as with the usual Euclidean algorithm.
\end{proof}

\begin{subalgorithm}[\texttt{LocalGenerator}]\label{locgene}
\begin{algorithmic}[1]
  \INPUT an integral right~$\order$-ideal~$I$ and a prime~$\prm$, for some maximal order~$\order$.
  \OUTPUT an element~$x\in I$ such that~$I_\prm = x\order_\prm$.
  \STATE $b_1,\dots,b_n \sto $ an $LLL$-reduced~$\Z$-basis of~$I$
  \STATE $e\sto v_\prm(\nrd(b_1))$
  \STATE $\ring\sto \Z_\fld/\prm^{e+1}$
  \STATE $B_1,\dots,B_n\sto $ images of~$b_1,\dots,b_n$ in~$\mat_2(\ring)$
  \STATE $D\sto B_1$\label{stepinit}
  \FOR{$i=1$ to~$n$}
    \STATE $D\sto $ \texttt{GCDMatrix}($D$, $B_i$)
  \ENDFOR
  \STATE let~$\mu_1,\dots,\mu_n\in\Z$ be such that~$\sum\mu_i B_i=D$\label{stepsolve}
  \RETURN $\sum \mu_i b_i$
\end{algorithmic}
\end{subalgorithm}

\begin{proposition}
  Subalgorithm~\ref{locgene} is correct.
\end{proposition}
\begin{proof}
  By definition of the norm of an ideal, we have~$v_\prm(\nrd(I))\le v_\prm(\nrd(b_1)) = e$. Because of the choice of~$\ring$, at Step~\ref{stepinit} we have~$\det(D)\neq 0$, so the calls to \texttt{GCDMatrix} are valid. By the properties of Subalgorithm~\ref{gcdmat}, at Step~\ref{stepsolve} we have~$D\mat_2(\ring) = B_1\mat_2(\ring) + \dots + B_1\mat_2(\ring)$ so the integers~$\mu_1,\dots,\mu_n$ exist. Now let~$x = \sum \mu_i b_i\in I$ be the output of the algorithm, so that~$v_\prm(\nrd(x))\le e$, hence~$\prm^{e+1}\order_\prm\subset x\order_\prm$. Let~$y\in I_\prm$. By reduction modulo~$\prm^{e+1}$ there exists~$a\in\order_\prm$ and~$b\in\prm^{e+1}\order$ such that~$y = xa+b\in x\order_\prm$, proving the result.
\end{proof}

Now we can present the local reduction algorithm.

\begin{subalgorithm}[\texttt{PReduce}]\label{preduce}
\begin{algorithmic}[1]
	\INPUT an integral right $\order$-ideal $I$, a prime $\prm$ and a $\prm$-reduction structure.
	\OUTPUT an integer~$r$, an element~$c\in \alg^\times$ and an integral~$\order$-ideal~$J$ such that~$cI = J\prm^r$ and~$v_\prm(\nrd(J))\in\{0,1\}$.
	\STATE $r\sto$ largest integer such that~$I\subset \order\prm^r$, $J\sto I\prm^{-r}$\label{stepts}
	\STATE $k\sto v_\prm(\nrd(J))$
	\STATE $c \sto 1$, $b\sto 0$
	\STATE $x \sto$ \texttt{LocalGenerator}($I$, $\prm$)
	\STATE $Q\sto x\act P_0$
	\REPEAT
	  \STATE $P\sto $ point at distance~$1$ from~$P_b$ in the segment~$P_bQ$
	  \STATE $(c,J,Q)\sto g\act (c,J,Q)$ where~$g\in\order_b^\times$ is such that~$g\act P=P_{1-b}$\label{stepmul}
	  \LINEIF{$b=1$}{$J\sto J\prm^{-1}$, $r\sto r+1$, $k\sto k-2$}\label{stepdivide}
	  \STATE $b\sto 1-b$
	\UNTIL{$k<2$}
	\RETURN $J, c, r$
\end{algorithmic}
\end{subalgorithm}


  In Step~\ref{stepts}, we have~$2r = v_\prm(\nrd(\tsidl))$ where~$\tsidl$ is the two-sided $\order$-ideal generated by~$I$. We can compute~$\tsidl$ as follows: if~$w_1,\dots,w_n$ is a~$\Z$-basis of~$\order$ and~$b_1,\dots,b_n$ is a~$\Z$-basis of~$I$, then~$\tsidl = \sum_{i,j}\Z w_ib_j$.
  \medskip
  

\begin{figure}[thb]
\begin{center}
\begin{tikzpicture}[btt,
  level/.style={
    level distance/.expanded=\ifnum#1>1 \tikzleveldistance/1.5\else\tikzleveldistance\fi,
    nodes/.expanded={\ifodd#1 fill=black\else fill=white\fi}
  }]
  \begin{scope}[xshift=-\scsh,yshift=\scsh]
  \draw (0,0) node[draw=none,above,minimum size=20pt]{$P_0$};
  \draw (0,-1) node[draw=none,below,minimum size=20pt]{$P_1$};
  \draw (180:2) node[draw=none]{$Q$};
  \draw (210:1.15) node[draw=none]{$g$};
  \draw[->, >=stealth] (165:1) arc (165:255:1);
  \draw (-67:2.15) node[draw=none]{$gQ$};
  \draw (-1.9,-1.7) node[draw=none]{$k=4$};
  \draw (-1.9,-1.95) node[draw=none]{$b=0$};
  \path[rotate=30]
  node[scale=\markscale]{}
  child[dotted] {
    node[solid,scale=\markscale]{}
    child[solid]{
      node{}
      \levelt
    }
    child[dotted]{
      node[solid]{}
      child[dotted]{
	node[solid]{}
	child[solid]{
	  node[solid]{}
	  \levelc
	}
	child[dotted]{
	  node[solid,scale=\markscale]{}
	  \levelc
	}
      }
      child[solid]{
	node[solid]{}
	\levelq
      }
    }
  }
  child {
    node{}
    \leveld
  }
  child[dashed] {
    node[solid]{}
    child[solid] {
      node[solid] {}
	\levelt
    }
    child[dashed] {
      node[solid]{}
      child[solid]{
	node[solid] {}
	\levelq
      }
      child[dashed]{
	node[solid] {}
	child[densely dashed]{
	  node[solid, scale=\markscale]{}
	  \levelc
	}
	child[solid]{
	  node{}
	  \levelc
	}
      }
    }
  };
  \end{scope}

  \begin{scope}[xshift=\scsh,yshift=\scsh]
  \draw (0,0) node[draw=none,above,minimum size=20pt]{$P_0$};
  \draw (0,-1) node[draw=none,below,minimum size=20pt]{$P_1$};
  \draw (-68:2.15) node[draw=none]{$Q$};
  \draw (39:1.6) node[draw=none]{$gQ$};
  \draw (-33:1) node[draw=none]{$g$};
  \draw[->, >=stealth] (0.6,-1.2) arc (-35:75:0.7);
  \draw (-1.9,-1.7) node[draw=none]{$k=4$};
  \draw (-1.9,-1.95) node[draw=none]{$b=1$};
  \path[rotate=30]
  node[scale=\markscale]{}
  child[dotted] {
    node[solid,scale=\markscale]{}
    child[solid]{
      node[solid]{}
      \levelt
    }
    child[dashed]{
      node[solid]{}
      child[dashed]{
	node[solid]{}
	child[solid]{
	  node[solid]{}
	  \levelc
	}
	child[densely dashed]{
	  node[solid, scale=\markscale]{}
	  \levelc
	}
      }
      child[solid]{
	node[solid]{}
	\levelq
      }
    }
  }
  child[dotted] {
    node[solid]{}
    child[solid]{
      node{}
      \levelt
    }
    child[dotted]{
      node[solid, scale=\markscale]{}
      \levelt
    }
  }
  child {
    node{}
    \leveld
  };
\end{scope}

\begin{scope}[xshift=-\scsh,yshift=-\scsh]
  \draw (0,0) node[draw=none,above,minimum size=20pt]{$P_0$};
  \draw (0,-1) node[draw=none,below,minimum size=20pt]{$P_1$};
  \draw (39:1.6) node[draw=none]{$Q$};
  \draw (-115:1.35) node[draw=none]{$gQ$};
  \draw[<-, >=stealth] (-75:1) arc (-75:15:1);
  \draw (-30:1.15) node[draw=none]{$g$};
  \draw (-1.9,-1.7) node[draw=none]{$k=2$};
  \draw (-1.9,-1.95) node[draw=none]{$b=0$};
  \path[rotate=30]
  node[scale=\markscale]{}
  child[dotted] {
    node[solid,scale=\markscale]{}
    child[dotted]{
      node[solid, scale=\markscale]{}
      \levelt
    }
    child[solid]{
      node{}
      \levelt
    }
  }
  child[dashed]{
    node[solid]{}
    child[solid]{
      node[solid]{}
      \levelt
    }
    child[dashed]{
      node[solid, scale=\markscale]{}
      \levelt
    }
  }
  child {
    node{}
    \leveld
  };
\end{scope}

\begin{scope}[xshift=\scsh,yshift=-\scsh]
  \draw (0,0) node[draw=none,above,minimum size=20pt]{$P_0$};
  \draw (0,-1) node[draw=none,below,minimum size=20pt]{$P_1$};
  \draw (-112:1.85) node[draw=none]{$Q$};
  \draw (215:1) node[draw=none]{$g$};
  \draw[->, >=stealth] (-0.6,-1.2) arc (215:105:0.7);
  \draw (-1.9,-1.7) node[draw=none]{$k=2$};
  \draw (-1.9,-1.95) node[draw=none]{$b=1$};
  \path[rotate=30]
  node[scale=\markscale]{}
  child[dotted] {
    node[solid,scale=\markscale]{}
    child[dashed]{
      node[solid, scale=\markscale]{}
      \levelt
    }
    child[solid]{
      node{}
      \levelt
    }
  }
  child {
    node{}
    \leveld
  }
  child {
    node{}
    \leveld
  };
\end{scope}
\end{tikzpicture}
\end{center}
\caption{\texttt{PReduce} (Subalgorithm~\ref{preduce})}
\label{figred}
\end{figure}

\begin{proposition}\label{proppreduce}
  Subalgorithm~\ref{preduce} is correct.
\end{proposition}
\begin{proof}
  Since~$k$ decreases by~$2$ every two iterations and is positive by the loop condition, the algorithm terminates. We now prove that the output is correct.
  \par First, the distance~$d(P_b,Q)$ decreases by~$1$ during each execution of the loop: before Step~\ref{stepmul}, we have~$d(P_{1-b},gQ)=d(gP,gQ)=d(P,Q)=d(P_b,Q)-1$ since~$P$ is at distance~$1$ from~$P_b$ on the segment~$P_bQ$. We claim that before or after a complete execution of the loop, we have~$d(P_0,Q)=k$ (see Figure~\ref{figred}). The claim is true before the first iteration: we have~$x\in\order\setminus\prm\order$ so~$d(P_0,Q)=d(P_0,xP_0)=v_\prm(\nrd(x))=v_\prm(\nrd(J))=k$. We only need to prove that the equality~$d(P_0,Q)=k$ is preserved when~$b=1$. In that case, $P_1$ is in the segment~$P_0Q$ because of the previous iteration, so that~$d(P_0,Q)=1+d(P_1,Q)=k$.
  
  \par We now prove that before or after a complete execution of the loop, the $\order$-ideal~$J$ is integral and~$v_\prm(\nrd(J))=k$. This property clearly holds before the first iteration. Before Step~\ref{stepdivide}, after two iterations~$b=0$ and~$b=1$, $J$ and~$Q$ have been multiplied by an element~$h$ such that~$v_\prm(\nrd(h))=0$ and~$d(P_0,hQ)=d(P_0,Q)-2$, so~$hJ$ is divisible by~$\prm$. Step~\ref{stepdivide} hence preserves integrality and updates~$k$ according to the valuation of~$\nrd(J)$.
  \par We now prove the proposition. The element~$c$ is a product of elements of~$\order_0^\times$ and~$\order_1^\times$ so~$c$ is a~$\prm$-unit with~$\nrd(c)\in\Z_\fld^\times$. We have just proved that~$J$ is integral, and by Step~\ref{stepdivide} the value of~$r$ is such that~$cI = J\prm^r$. Because of the loop condition, after the algorithm terminates we have~$v_\prm(\nrd(J))=k\in\{0,1\}$.
\end{proof}

We now explain how to perform global reduction. We use linear algebra to control the valuations of a smooth ideal and then perform local reduction at every prime to get an ``almost two-sided ideal''. The first step is similar to its commutative analogue: we need sufficiently many ``relations'' (smooth elements in~$\alg^\times$) so that the quotient of the factor base by the norms of the relations is the ray class group~$\clalg$. This leads to the definition of a G-reduction structure.

\begin{definition}\label{defgred}
  A \emph{G-reduction structure} is given by the following data:
  \begin{enumerate}
    \item a~$\prm$-reduction structure for each~$\prm\in\fb$ that splits in~$\alg$;
    \item\label{defgred2} a finite set of elements~$X\subset \order\cap \alg^\times$ and a map~$\phi:\clalg \rightarrow \gp{\fb}$ that is a lift of an isomorphism~$\clalg \xrightarrow{\sim} \gp{\fb}/\gp{\nrd(X)}$ and such that~$\phi(1)=\Z_\fld$.
  \end{enumerate}
\end{definition}

The following algorithm performs global reduction. In order to avoid explosion of the size of the ideal in the local reduction, we extract the two-sided part, allowing us to reduce all exponents modulo~$2$. The remaining part stays small and gets $\prm$-reduced, while the two-sided part is only multiplied by powers of primes.

\begin{subalgorithm}[\texttt{GReduce}]\label{greduce}
\begin{algorithmic}[1]
	\INPUT a smooth integral right $\order$-ideal~$I$ and a G-reduction structure.
	\OUTPUT an integral ideal $J$, an element~$c\in \alg^\times$ and a two-sided ideal $\tsidl$ such that~$cI=J\tsidl$ and~$I$ is principal if and only if~$J=\order$ and~$\tsidl=\order$.
	\STATE $\idl\sto\nrd(I)$
	\STATE $\idlb\sto\phi(\idl)$ where~$\idl$ is seen as an element of~$\clalg$
	\STATE let~$e\in\Z^X$ be such that~$\nrd(y)\idl = \idlb$ where~$y = \prod_{x\in X}x^{e_x}$.
	\STATE $c\sto \prod_{x\in X}x^{e_x \mathrm{mod}\, 2}$, $f\sto \prod_{x\in X}\nrd(x)^{\floor{e_x/2}}$ \COMMENT{Extract two-sided part}
	\STATE $J\sto cI$\label{stepgmul}
	\STATE $\tsidl\sto$ two-sided ideal generated by~$J$
	\STATE $J\sto J\tsidl^{-1}$ \COMMENT{Extract two-sided part}\label{stepgts}
	\STATE $\tsidl\sto f\tsidl$
	\FOR{$\prm\in\fb$ dividing~$\nrd(J)$ and splitting in~$\alg$}\label{stepgloop}
	  \STATE $J,c',r \sto$ \texttt{PReduce}($J, \prm$)
	  \STATE $c\sto c'c$, $\tsidl\sto\prm^r\tsidl$
	\ENDFOR
	\RETURN $J, cf, \tsidl$
\end{algorithmic}
\end{subalgorithm}

\begin{proposition}\label{propgreduce}
  Subalgorithm~\ref{greduce} is correct.
\end{proposition}
\begin{proof}
  Since Step~\ref{stepgts} and \texttt{PReduce} preserve integrality (Proposition~\ref{proppreduce}), the output~$J$ is integral. The relation~$cI=J\tsidl$ is clear by tracking the multiplications. If the output is such that~$J=\order$ and~$\tsidl=\order$, then~$I=c^{-1}\order$ is principal. Conversely, if~$I$ is principal, then~$\idl=\nrd(I)$ is trivial in the class group~$\clalg$ so~$\idlb=\phi(\mathrm{cl}(\idl))=\Z_F$. After Step~\ref{stepgmul}, we have~$\nrd(J)=f^{-2}\Z_F$. After Step~\ref{stepgts}, we have multiplied~$J$ by a two-sided ideal, so~$v_\prm(\nrd(J))$ is even for all primes~$\prm$ splitting in~$\alg$. Since~$J$ is not divisible by a two-sided ideal, $\nrd(J)$ is not divisible by primes that ramify in~$\alg$. We obtain~$J=\order$ at the end of the loop by the properties of~\texttt{PReduce} so~$\nrd(\tsidl)=\Z_F$. Since~$\tsidl$ is two-sided, it is entirely determined by its norm so~$\tsidl=\order$.
\end{proof}

Finally, we reduce the general case to the smooth case by the noncommutative analogue of standard randomizing techniques. We generate a random smooth $\order$-ideal by the following procedure, to which we refer as \texttt{RandomLeftIdeal}($\order$). For each~$\prm\in\fb$, pick a nonnegative integer~$k$. Let~$\iota:\order\to\mat_2(\Z_F/\prm^k)$ be a splitting map. Let~$M\in\mat_2(\Z_F/\prm^k)$ be a random upper-triangular matrix with zero determinant and compute~$R_\prm = \order\iota^{-1}(M)+\prm^k\order$. Finally, return~$\bigcap_\prm R_\prm$. Choose the exponents~$k$ such that~$\na(R)\approx\Delta$.
\par It not clear at the moment what the best distribution for the exponents is.
A simpler idea would be to use random products~$\prod_\prm \prm^{k_\prm}$. In our experience, this leads to poorly randomized ideals. This is clear in the case~$\fld=\Q$: the randomized ideals are simply integer multiples of~$\order$.

\begin{lmsalgorithm}[\texttt{IsPrincipal}] \label{isprincipal}
\begin{algorithmic}[1]
	\INPUT an integral right $\order$-ideal~$I$ and a G-reduction structure.
	\OUTPUT an integral ideal $J$, an element~$c\in \alg^\times$ and a two-sided ideal $\tsidl$ such that~$cfI=J\tsidl$ and~$I$ is principal if and only if~$J=\order$ and~$\tsidl=\order$.
	  \STATE $R\sto$ \texttt{RandomLeftIdeal}($\order$)
	  \STATE $x\sto$ \texttt{NextElement}($I^{-1}\cap R$)
	\LINEIF{$xI$ is not smooth}{ \algorithmicreturn \textbf{ FAIL }}\label{stepsmooth}
	\STATE $J, c, \tsidl \sto$ \texttt{GReduce}($xI$)
	\RETURN $J, cx, \tsidl$
\end{algorithmic}
\end{lmsalgorithm}

By Proposition~\ref{propgreduce}, if Algorithm~\ref{isprincipal} does not return FAIL, its output is correct. In practice, we repeat Algorithm~\ref{isprincipal} until it returns the result.

\subsection{Building the reduction structures}\label{secstruct}
Now we explain how to build the previous reduction structures. The local reduction structure needs units in~$\order$. In general it is difficult to compute the whole unit group~$\order^\times$: for instance in the Fuchsian case, the minimal number of generators is at least~$\Delta^{3/8+o(1)}$ (this follows from the theory of signatures of Fuchsian groups~\cite[Section 4.3]{katok} and a volume formula~\cite[Theorem 11.1.1]{macreid}), which makes it hopeless to find a subexponential method. However, we can find some units in~$\order$ by considering commutative suborders and computing generators of their unit group with Buchmann's algorithm. Heuristically, these units are sufficiently random for our purpose. This strategy is implemented by the following algorithms.

\begin{subalgorithm}[\texttt{P1Search}] \label{p1search}
\begin{algorithmic}[1]
	\INPUT a maximal order~$\orderb$ and a prime~$\prm$.
	\OUTPUT a set of elements~$X\subset\orderb^\times$ acting transitively on~$\Proj^1(\fldfin{\prm})$.
	\STATE $X\sto\emptyset$
	\REPEAT
	  \STATE $x\sto$ \texttt{NextElement}($\orderb$)\label{stepnexteltp1s}
	  \STATE $\fldext\sto \fld(x)$
	  \IF{$\fldext/\fld$ has positive relative unit rank}\label{stepposrk}
	    \STATE $\ring\sto\Z_\fldext\cap\orderb$
	    \STATE\label{stepunits} $X\sto X\,\cup$ a set of generators of $\ring^\times$
	  \ENDIF
	\UNTIL{$X$ acts transitively on~$\Proj^1(\fldfin{\prm})$}
	\RETURN $X$
\end{algorithmic}
\end{subalgorithm}

In Step~\ref{stepunits}, we can compute the unit group~$\ring^\times$ with the algorithms of Kl\"uners and Pauli~\cite{rings}. Note that we actually do not need the full group~$\ring^\times$: a subgroup of finite index is sufficient.


\begin{proposition}\label{propp1search}
  Subalgorithm~\ref{p1search} is correct.
\end{proposition}
\begin{proof}
  By strong approximation (Theorem~\ref{thmstrongapprox}), the group~$\orderb^\times$ acts transitively on~$\Proj^1(\fldfin{\prm})$. This group is finitely generated, so after finitely many iterations we will have enumerated a set of generators and the algorithm will terminate. By the loop condition, the output is correct.
\end{proof}

\begin{subalgorithm}[\texttt{PBuild}] \label{pbuild}
\begin{algorithmic}[1]
	\INPUT a maximal order~$\order$ and a prime~$\prm$.
	\OUTPUT a $\prm$-reduction structure.
	\STATE $I\sto$ an integral right $\order$-ideal of norm~$\prm$
	\STATE $\order_1\sto\order_l(I)$
	\STATE $X\sto$ \texttt{P1Search}($\order, \prm$)
	\STATE from~$X$, for each~$P$ at distance~$1$ from~$P_0$ compute an element~$g\in\order^\times$ such that~$g\act P=P_1$%
	\STATE $X\sto$ \texttt{P1Search}($\order_1, \prm$)\label{stepo1p1search}
	\STATE from~$X$, for each~$P$ at distance~$1$ from~$P_1$ compute an element~$g\in\order_1^\times$ such that~$g\act P=P_0$%
	\RETURN the $\prm$-reduction structure
\end{algorithmic}
\end{subalgorithm}

\begin{remark}
  Let~$g,g'\in\order^\times$ be such that~$g\act P_1 = g'\act P_1$ and let~$h=g^{-1}g'$. Then~$h\act P_1=P_1$ so~$h\in\order^\times\cap\order_1^\times$. This allows us to construct many elements in~$\order_1^\times$ before Step~\ref{stepo1p1search}. If we have sufficiently many such units, which often happens in practice, they will act transitively on the points~$P\neq P_0$ at distance~$1$ from~$P_1$. In this case, in Step~\ref{stepo1p1search} we will only need to find one element~$g\in\order_1^\times$ such that~$g\act P_0\neq P_0$.
\end{remark}

We build the global reduction structure in a way similar to the commutative case: we look for small relations in smooth ideals. In addition, we get a good starting point thanks to the inclusion~$\Z_\fld\subset\order$: the units~$\Z_{\fld,\fb}^\times$ provide all the relations up to a~$2$-elementary Abelian group.

\begin{lmsalgorithm}[\texttt{GBuild}] \label{gbuild}
\begin{algorithmic}[1]
	\INPUT a maximal order~$\order$ and a factor base~$\fb$.
	\OUTPUT a G-reduction structure.
	\FOR{$\prm\in\fb$ that splits in~$\alg$}
	  \STATE \texttt{PBuild}($\order, \prm$)\label{steppbuild}
	\ENDFOR
	\STATE $X\sto$ integral generators of~$\Z_{\fld,\fb}^\times$\label{stepsunits}
	\FOR{$\prm\in\fb$}
	  \STATE $I\sto $ integral $\order$-ideal of norm~$\prm$\label{stepidnormp}
	  \REPEAT\label{loop1}
	    \STATE $x\sto$ \texttt{NextElement}($I$)\label{stepnextgb1}
	  \UNTIL{$x$ is smooth}
	  \STATE $X\sto X\cup\{x\}$
	\ENDFOR
	\WHILE{$\gp{\fb}/\gp{\nrd(X)}\not\cong\clalg$}\label{loop2}
	  \STATE $x\sto$ \texttt{NextElement}($\order$)\label{stepnextgb2}
	  \LINEIF{$x$ is smooth}{$X\sto X\cup\{x\}$}
	\ENDWHILE
	\RETURN the G-reduction structure
\end{algorithmic}
\end{lmsalgorithm}

\begin{remark}
  The various calls to~\texttt{PBuild} in Step~\ref{steppbuild} are not completely independent: we can keep the elements in~$\order^\times$ from one call for other ones.
\end{remark}

\begin{proposition}\label{propgbuild}
  Algorithm~\ref{gbuild} is correct.
\end{proposition}
\begin{proof}
  Let~$I$ be the ideal in Step~\ref{stepidnormp}. There exists a smooth ideal~$J$ equivalent to~$I$, let~$x\in\alg^\times$ be such that~$I=xJ$. Then~$x\order = IJ^{-1}$, so~$\nrd(J)x\in I\bar{J}\subset I$ is smooth. It will be enumerated at some point, so the loop starting at Step~\ref{loop1} terminates. Since~$\fb$ generates the class group~$\clalg$, by Eichler's theorem (Theorem~\ref{thmeichlerclass}) we have~$\gp{\fb}/\nrd(\alg^\times)\cong\clalg$, so there exists a finite set of~$\fb$-smooth elements~$X\subset\order$ such that~$\gp{\fb}/\gp{\nrd(X)}\cong\clalg$. We will enumerate this set at some point, so the loop starting at Step~\ref{loop2} terminates. So Algorithm~\ref{gbuild} terminates, and by Proposition~\ref{propp1search} and Step~\ref{loop2} it returns a correct G-reduction structure.
\end{proof}

\begin{remark}
  We have restricted to maximal orders to simplify the exposition, but this restriction can be weakened as follows. Let~$\order$ be an arbitrary order, and let~$S$ be the set of primes~$\prm$ of~$\Z_\fld$ such that~$\order$ is not $\prm$-maximal. The set $S$ is finite, so we can choose a factor basis disjoint from~$S$. Then our algorithms work unchanged for right $\order$-ideals except one point: Theorem~\ref{thmeichlerclass} characterizing principal ideals might no longer hold. If we restrict to Eichler orders, that is intersections of two maximal orders, Theorem~\ref{thmeichlerclass} still holds. Otherwise we need to find the suitable class group and change Definition~\ref{defgred} \ref{defgred2} accordingly.
  %
\end{remark}

\subsection{Compact representations}\label{seccr}
In the previous algorithms, the cost of elementary operations is important. Representing units as linear combinations of a basis of~$\order$ could be catastrophic: the classical example of real quadratic fields suggests that fundamental units in commutative orders, such as those produced by Subalgorithm~\ref{p1search}, can have exponential size. This problem is classically circumvented by representing units in \emph{compact form}: a product of small~$S$-units with possibly large exponents. The problem is reduced to computing efficiently with those compact representations. %
A natural notion of compact representation in~$\order$ would be to take ordered products of $S$-units in~$\order$ but we do not know how to compute efficiently with such a general representation. Instead we use a more restricted notion: 
we group the units belonging to a common commutative suborder, in which we can compute efficiently. This leads to the following definition.

\begin{definition}
  A \emph{compact representation} in~$\order$ is:
  \begin{enumerate}
    \item an element~$x\in\order$, or
    \item\label{typetwo} a product~$y = \prod_{i=1}^{r}y_i^{e_i}$ where the exponents are signed integers, the elements~$y_i$ all lie in a single ring~$\ring\subset \order$ containing~$\Z_\fld$ and such that~$y\in\ring^\times$, together with a~$\Z$-basis of the integral closure of~$\ring$ and a factorization of its conductor, or
    \item an ordered product of compact representations.
  \end{enumerate}
  A product~$y$ as in~\ref{typetwo} will be called a \emph{representation of type~\ref{typetwo}}.
\end{definition}

We describe the algorithms for representations of type~\ref{typetwo}, and they naturally extend by multiplicativity to arbitrary compact representations. We first explain the algorithm for local evaluation of compact representations. Since the product represents a unit, we can replace it with a product of local units, avoiding loss of precision despite the large exponents.
\begin{subalgorithm}[\texttt{EvalCR}] \label{evalcr}
\begin{algorithmic}[1]
	\INPUT a representation~$y = \prod_{i=1}^{r}y_i^{e_i}\in\ring$ of type~\ref{typetwo}, an ideal~$\idl\subset\Z_\fld$.
	\OUTPUT an element~$z\in\order$ such that~$z=y\pmod{\idl\order}$.
	\STATE $\fldext\sto$ field of fractions of~$\ring$
	\STATE $\cnd\sto$ conductor of~$\ring$ inside~$\Z_\fldext$
	\STATE $\prod_{j=1}^{k}\Prm_j^{w_j}\sto$ factorization of~$\idl\cnd\Z_\fldext$
	\FOR{$j=1$ to $k$}
	  \STATE $\phi \sto$ reduction map onto~$\Z_\fldext/\Prm_j^{w_j}$
	  \STATE $\pi\sto$ uniformizer in~$\Prm_j$, $v\sto v_{\Prm_j}$
	  \STATE $z_j\sto \prod_{i=1}^{r}\phi(y_i\pi^{-v(y_i)})^{e_i}$
	\ENDFOR
	\STATE $z\sto$ element in~$\Z_\fldext$ such that~$z=z_j\pmod{\Prm_j^{w_j}}$
	\RETURN $z$
\end{algorithmic}
\end{subalgorithm}

\begin{proposition}\label{propevalcr}
  Subalgorithm~\ref{evalcr} is correct.
\end{proposition}
\begin{proof}
  First, we claim that for all~$j\le k$, we have~$z_j=y\pmod{\Prm_j^{w_j}}$. Since~$\nrd(y)\in\Z_\fld^\times$, we have~$\sum_{i=1}^r v(y_i)e_i=0$ so~$\prod_{i=1}^{r}(y_i\pi^{-v(y_i)})^{e_i} = \prod_{i=1}^{r}\pi^{-v(y_i)e_i}\prod_{i=1}^{r}y_i^{e_i}=y$. Since~$y_i\pi^{-v(y_i)}$ is integral at~$\Prm_j$, we can apply~$\phi$ to it and the claim follows. This implies that the output~$z$ of the algorithm satisfies~$z=y\pmod{\idl\cnd\Z_\fldext}$. In particular, we have~$z=y\pmod{\cnd\Z_\fldext}$ so~$z-y\in\ring$. Since~$y\in\ring$ we get~$z\in\ring\subset\order$. The relation~$\idl\cnd\Z_\fldext\subset\idl\ring\subset\idl\order$ shows that~$z=y\pmod{\idl\order}$.
\end{proof}

We can now explain how to multiply an ideal by a compact representation. To know an ideal, it suffices to know it up to large enough precision: we reduce the problem to local evaluation.
\begin{subalgorithm}[\texttt{MulCR}] \label{mulcr}
\begin{algorithmic}[1]
	\INPUT a representation~$y = \prod_{i=1}^{r}y_i^{e_i}\in\ring$ of type~\ref{typetwo}, an integral right~$\order$-ideal~$I$.
	\OUTPUT the ideal~$yI$.
	\STATE $\idl\sto \nrd(I)$
	\STATE $z\sto$ \texttt{EvalCR}($y, \idl$)
	\RETURN $zI+\idl\order$
\end{algorithmic}
\end{subalgorithm}

\begin{proposition}\label{propmulcr}
  Subalgorithm~\ref{mulcr} is correct.
\end{proposition}
\begin{proof}
  By Proposition~\ref{propevalcr}, we have~$z=y\pmod{\idl\order}$ so~$(z-y)I\subset(z-y)\order\subset\idl\order$, which gives~$zI+\idl\order=yI+\idl\order$. Since~$y\in\order^\times$, we have~$\nrd(yI)=\nrd(I)=\idl$, so~$\idl\order\subset yI$ and finally~$yI+\idl\order=yI$. Therefore the output of the algorithm is correct.
\end{proof}

\section{Complexity analysis}\label{seccomplexity}
We perform a complete complexity analysis of our algorithms, assuming suitable heuristics. To simplify the notations, we set~$L(x)=\exp(\sqrt{\log x\log\log x})$. In every complexity estimate, the degree of the base field~$\fld$ is fixed. When we mention a complexity of the form~$\secompl$, we always implicitly mean~$\secompl$ times a polynomial in the size of the input.  We fix a parameter~$\alpha>0$. We will analyse our algorithm using the general paradigm that with a factor base of subexponential size, elements have a subexponential probability of being smooth. However, recall from Definition~\ref{deffb} that the factor base~$\fb$ is assumed to generate the ray class group~$\clalg$, so we need the following heuristic.

\begin{heuristic}\label{heurfb}
There exists a constant~$c=c_\alpha$ such that for every quaternion algebra~$\alg$ with absolute discriminant~$\Delta$, the set of primes having norm less than~$c\cdot L(\Delta)^{\alpha}$ generate the class group~$\clalg$.
\end{heuristic}

This heuristic is a theorem under the Generalized Riemann Hypothesis~\cite{bach}.  By Minkowski's bound, Heuristic~\ref{heurfb} is also true unconditionnally with the restriction that~$\log\Delta\gg_{\alpha} (\log |\disc_\fld|)^2$. 
From now on, we assume Heuristic~\ref{heurfb} and we assume that the factor base~$\fb$ is the set of primes having norm less than~$c\cdot L(\Delta)^{\alpha}$. Note that this implies the bound~$\#\fb \le L(\Delta)^{\alpha+o(1)}$.

\medskip

We start by analysing the complexity of elementary operations: Subalgorithm~\ref{locgene} and the algorithms of Section~\ref{seccr}.

\begin{lemma}\label{lemcompllocgene}
  Subalgorithm~\ref{locgene} terminates in time polynomial in the size of the input.
\end{lemma}
\begin{proof}
  Subalgorithm~\ref{divmat} (\texttt{DivideMatrix}) is made of a constant number of elementary operations, so it is polynomial. In Subalgorithm~\ref{gcdmat} (\texttt{GCDMatrix}) with~$\ring = \Z_\fld/\prm^i$, since the valuation of the determinant decreases at every recursive call, there are at most~$i$ such calls. When Subalgorithm~\ref{locgene} (\texttt{LocalGenerator}) calls Subalgorithm~\ref{gcdmat}, we have~$i=e+1=v_{\prm}(\nrd(b_1))+1 = O(\log \na(I))$ by lattice reduction. So the algorithm is polynomial in the size of the input.
\end{proof}

\begin{lemma}\label{lemcomplcr}
  Given the factorization of~$\idl$, Subalgorithm~\ref{evalcr} terminates in time polynomial in the size of the input. Given the factorization of~$\nrd(I)$, Subalgorithm~\ref{evalcr} terminates in time polynomial in the size of the input.
\end{lemma}
\begin{proof}
 In Subalgorithm~\ref{evalcr}, the number of iterations of the loop is polynomial in the size of the input. The only operations that could possibly not be polynomial in the size of the input are the computation of~$\Z_\fldext$ and the factorization of~$\idl\cnd\Z_L$, but~$\Z_\fldext$ and the factorization of~$\cnd$ are contained in the compact representation, and the factorization of~$\idl$ is assumed to be given. So the algorithm is polynomial in the size of its input. Since the HNF of the output can be computed in polynomial time~\cite{hnf}, Subalgorithm~\ref{evalcr} is also polynomial.
\end{proof}

The restriction on the factorization is not a problem in our application: every ideal on which we call these algorithms is smooth.

\medskip

Since our algorithms use their commutative counterparts, we have to make assumptions on the algorithms used to compute commutative unit groups.
\begin{heuristic}\label{heurfastunits}
  There is an explicit algorithm that, given a number field~$\fldb$ with discriminant~$\disc_\fldb$, an order~$\ring$ in~$\fldb$ and a bound~$b = L(\disc_\fldb)^{O(1)}$, computes a set~$U$ of integral generators for the~$S$-unit group of~$\Z_\fldb$, where~$S$ is the set of primes of norm less than~$b$, and generators for the unit group~$\ring^\times$ expressed as products of elements in~$U$, in expected time~$L(\disc_\fldb)^{O(1)}$.
\end{heuristic}
This is a strong hypothesis. However, under the Generalized Riemann Hypothesis it is a theorem for quadratic number fields~\cite{hmc,voll} and experience has shown that it is not an unreasonable assumption\footnote{The PARI developers experimented extensively with this algorithm in the past twenty years, as implemented in the PARI/GP function \texttt{bnfinit}, while building  and checking tables of  number fields of small degree [Ref: \url{http://pari.math.u-bordeaux1.fr/pub/pari/packages/nftables/}], as well as with number fields of much larger degree. E.g. the class group and unit group of $\fld = \Q[t]/(t^{90}-t^2-1)$, $d_\fld > 10^{175}$ can be computed in a few hours (Karim Belabas, personal communication).}.


\medskip

Since the reduction algorithms depend on the structures that are given as input, we analyse the algorithms building the reduction structures before the reduction algorithms. We start with Subalgorithm~\ref{p1search}, for which we need some heuristic assumptions.

\begin{heuristics}\label{heurunits} In Subalgorithm~\ref{p1search}, we assume the following.
\begin{enumerate}[(i)]
  \item\label{heurunits1} If~$\fld$ is totally real, a positive proportion of~$x$ satisfies the condition of Step~\ref{stepposrk} that~$\fld(x)/\fld$ has positive relative unit rank.
  \item\label{heurunits2} The images in~$\PGL_2(\fldfin{\prm})$ of the units produced at Step~\ref{stepunits} are uniformly distributed in the image of~$\order^\times$ in~$\PGL_2(\fldfin{\prm})$.
\end{enumerate}
\end{heuristics}

%

\begin{proposition}\label{propcomplp1search}
  Assuming Heuristics~\ref{heurfastunits} and~\ref{heurunits}, the expected running time of Subalgorithm~\ref{p1search} is at most~$\secompl$.
\end{proposition}
\begin{proof}
  We first prove that the expected number of iterations of the loop is~$O(1)$. If $\fld$ is not totally real, the relative unit rank condition is always satisfied, so by Heuristic~\ref{heurunits}~(\ref{heurunits1}) a positive proportion of the iterations of the loop produce a unit.
  By strong approximation, the image of~$\order^\times$ in~$\PGL_2(\fldfin{\prm})$ contains~$\PSL_2(\fldfin{\prm})$, and the index is at most~$2$. By Heuristic~\ref{heurunits}~(\ref{heurunits2}), with probability at least~$1/2$ the image of the unit produced at Step~\ref{stepunits} is in~$\PSL_2(\fldfin{\prm})$, and the corresponding images are equidistributed in~$\PSL_2(\fldfin{\prm})$. By~\cite{lubo}, the probability that two random elements of~$\PSL_2(\fldfin{\prm})$ generate this group is bounded below by a constant. 
  Therefore, after an expected number of iterations~$O(1)$, the image of~$X$ generates~$\PSL_2(\fldfin{\prm})$ and hence acts transitively on~$\Proj^1(\fldfin{\prm})$.
  \par Each computation of a unit group in Step~\ref{stepunits} takes expected time~$\secompl$ by Heuristic~\ref{heurfastunits} since the discriminant of~$\Z_\fld[x]$ is~$\Delta^{O(1)}$. The units are stored in compact representation and we use Subalgorithm~\ref{evalcr} (\texttt{EvalCR}) to compute the action on~$\Proj^1(\fldfin{\prm})$, so by Lemma~\ref{lemcomplcr} this takes total time~$\secompl$.
\end{proof}


\begin{heuristics}\label{heursmoothdist} Let~$\Z_{\fld,\fb,\alg}^\times$ be the group of~$\fb$-units in~$\Z_\fld$ that are positive at every real place ramified in~$\alg$. In Algorithm~\ref{gbuild}, we assume the following.
\begin{enumerate}[(i)]
  \item\label{heursmoothdist1} There exists a constant~$\beta>0$ such that the elements~$x$ produced in Steps~\ref{stepnextgb1} and~\ref{stepnextgb2} are smooth with probability at least~$L(\Delta)^{-\beta+o(1)}$.
  \item\label{heursmoothdist2} The norms of the smooth elements produced in Step~\ref{stepnextgb2} are equidistributed in~$\Z_{\fld,\fb,\alg}^\times / \Z_{\fld,\fb}^{\times 2}$.
\end{enumerate}
\end{heuristics}

%

By comparison with the case of integers~\cite[Equation (1.16) and Section 1.3]{granville}, $\beta=1/(2\alpha)$ could be a reasonable value.


\begin{theorem}\label{thmcomplgbuild}
  Assume Heuristics~\ref{heurfastunits},~\ref{heurunits}, \ref{heurfb} and~\ref{heursmoothdist}. Then, given a maximal order~$\order$ in an indefinite quaternion algebra~$\alg$, Algorithm~\ref{gbuild} (\texttt{GBuild}) terminates in expected time~$\secompl$.
\end{theorem}
\begin{proof}
  There are~$2\cdot \#\fb = \secompl$ calls to Subalgorithm~\ref{p1search} (\texttt{P1Search}). By 	Proposition~\ref{propcomplp1search}, these calls take total time~$\secompl$. The computation of the group~$\Z_{\fld,\fb}^\times$ takes time~$L(\disc_\fld)^{O(1)}=\secompl$ by Heuristic~\ref{heurfastunits}. By Heuristic~\ref{heursmoothdist}~(\ref{heursmoothdist1}), the expected number of iterations in the loop starting at Step~\ref{loop1} is at most~$\secompl$ for each~$\prm\in\fb$, so the total expected number of iterations of this loop is at most~$\#\fb\cdot \secompl=\secompl$.
  \par We study the loop starting at Step~\ref{loop2}. After Step~\ref{stepsunits}, we have~$\gp{\fb}/\gp{\nrd(X)}=\gp{\fb}/\Z_{\fld,\fb}^{\times 2}$. Let~$C$ be the group~$\Z_{\fld,\fb,\alg}^\times / \Z_{\fld,\fb}^{\times 2}$. There is an exact sequence
  \[ 1 \longrightarrow C \longrightarrow \gp{\fb}/\Z_{\fld,\fb}^{\times 2} \longrightarrow \clalg \longrightarrow 1\text{,}\]
  so that the loop terminates if and only if the image of~$\nrd(X)$ generates the group~$C$. We have~$\# C = 2^{\#\fb + O(1)}$, so by Heuristic~\ref{heursmoothdist}~(\ref{heursmoothdist2}) this happens after we find an expected number of~$\#\fb+O(1)$ smooth elements. By Heuristic~\ref{heursmoothdist}~(\ref{heursmoothdist1}), the total expected number of iterations of this loop is at most~$\#\fb\cdot \secompl=\secompl$. Checking the loop condition~$\gp{\fb}/\gp{\nrd(X)}\not\cong\clalg$ amounts to linear algebra, so it also takes time~$\secompl$. This proves the theorem.
\end{proof}

\begin{theorem}
  Assume Heuristics~\ref{heurfastunits},~\ref{heurunits}, \ref{heurfb} and~\ref{heursmoothdist}. Then, given the G-reduction structure computed by Algorithm~\ref{gbuild} and a smooth integral right $\order$-ideal~$I$, Algorithm~\ref{greduce} (\texttt{GReduce}) terminates in expected time~$\secompl$.
\end{theorem}
\begin{proof}
  First, since by Theorem~\ref{thmcomplgbuild}, Algorithm~\ref{gbuild} terminates in expected time~$\secompl$ so in particular the expected size of its output is also at most~$\secompl$. In Subalgorithm~\ref{greduce} (\texttt{GReduce}), the first part is linear algebra so it takes time~$\secompl$. By Lemma~\ref{lemcomplcr}, all the elementary operations can be performed in time polynomial in the size of their input.

  We analyse the calls to Subalgorithm~\ref{preduce} (\texttt{PReduce}). In this subalgorithm, the variable~$k$ decreases by~$2$ every two iterations and the initial value of~$k$ is bounded by~$v_\prm(\nrd(J))$, so the algorithm terminates after at most~$v_\prm(\nrd(J))$ iterations by the loop condition. So the total number of iterations in the calls to Subalgorithm~\ref{preduce} is bounded by~$\sum_{\prm\in\fb}v_\prm(\nrd(J))\le \log_2\na(J) \le \log_2(N_X\cdot \na(I))$ by Step~\ref{stepgmul} of Subalgorithm~\ref{greduce}, where~$N_X = \prod_{x\in X}\na(x)$. But~$\log\na(I)$ is polynomial in the size of the input and~$\log N_X$ is polynomial in the size of the G-reduction structure, which is~$\secompl$. This proves the theorem.

\end{proof}

Finally, for a general integral right $\order$-ideal, repeated attempts with Algorithm~\ref{isprincipal} (\texttt{IsPrincipal}) takes total expected time~$\secompl$ if we assume the following heuristic.

\begin{heuristic}\label{heursmoothen}
  There exists a constant~$\gamma>0$ such that in Step~\ref{stepsmooth} of Algorithm~\ref{isprincipal}, the element~$x$ is smooth with probability at least~$L(\Delta)^{-\gamma+o(1)}$.
\end{heuristic}
Again, by comparison with the case of integers~\cite{granville}, $\gamma=1/(\sqrt{2}\alpha)$ could be a reasonable value.

\section{Examples}\label{secex}
We have implemented the above algorithms in the computer algebra system Magma~\cite{magma}. In this section, we demonstrate how our algorithms work and perform in practice. Every computation was performed on a $2.5$~GHz Intel Xeon E5420 processor with Magma v2.20-5 from the PLAFRIM experimental testbed.


\paragraph{Example 1.}Let~$\alg$ be the quaternion algebra over~$\Q$ generated by two elements~$i,j$ such that~$i^2=3$, $j^2=-1$ and~$ij=-ji$. The algebra~$\alg$ is ramified at~$2$ and~$3$ and unramified at every other place: $\alg$ is indefinite and our method applies. Let~$\order$ be the maximal order~$\Z+\Z i+\Z j+\Z\omega$ where~$\omega=(1+i+j+ij)/2$. We construct a reduction structure with Algorithm~\ref{gbuild} (\texttt{GBuild}) and factor base~$\fb=\{2,3,5,7,11,13,17\}$. Let~$I=19\order+a\order$ where~$a=-3-4i+j$, so that~$\nrd(I)=19\Z$. We use Algorithm~\ref{isprincipal} (\texttt{IsPrincipal}) to compute a generator of~$I$. It finds an element
~$x=(7+i-9j-3\omega)/19\in I^{-1}$ such that~$\nrd(xI)=7\Z$, so that~$xI$ is smooth. The linear algebra phase in Subalgorithm~\ref{greduce} (\texttt{GReduce}) computes
~$c=-1-2i-j+\omega$ having norm~$-7$ and~$f=1/7$. We obtain~$J=49\order+b\order$ with~$b=-17-8i+j$ and~$\tsidl=7^{-1}\order$ before the local reduction. We reduce the ideal~$J$ at~$7$. %
In Subalgorithm~\ref{preduce} (\texttt{PReduce}), at the first iteration we have~$P=P_1$ so~$c=1$. In the second iteration we have
~$c=(-9-5i-7j-3\omega)/7$: the element~$c$ has norm~$1$ and~$r=1$. After multiplying every element, we obtain the output
~$c=7/19\cdot(8+4i+3j-11\omega), f=1/7$ and
~$x=(cf)^{-1}=3+4i-3j-11\omega$ has norm~$-19$: $x$ is a generator of the ideal~$I$.



\paragraph{Example 2.}Let~$\fld$ be the complex cubic field of discriminant~$-23$, which is generated by an element~$t$ such that~$t^3-t+1=0$. Let~$\alg$ be the quaternion algebra over~$\fld$ generated by two elements~$i,j$ such that~$i^2=2t^2+t-3$,~$j^2=-5$ and~$ij=-ji$. The algebra~$\alg$ is ramified at the real place of~$\fld$ and the discriminant~$\delta_\alg$ has norm~$5$. All the maximal orders in~$\alg$ are conjugate and we compute one of them with Magma. Algorithm~\ref{gbuild} (\texttt{GBuild}) constructs the reduction structure in~$4$ seconds. We then compute the~$22$ primes of~$\fld$ coprime to~$\delta_\alg$ and having norm less than~$100$. For every such prime~$\prm$, we construct a random integral right $\order$-ideal~$I$ with norm~$\prm$. Since~$\clalg$ is trivial, they are all principal. We apply Algorithm~\ref{isprincipal} (\texttt{IsPrincipal}) to compute a generator of each of these ideals. %
This computation takes~$0.3$ seconds per ideal on average with a maximum of~$0.9$ seconds. As a comparison, we compute generators for the same ideals with the function provided by Magma. This computation takes~$4$~hours per ideal on average with a maximum of~$69$h, and~$5$ of the~$23$ ideals take less that~$0.1$~seconds.


\paragraph{Example 3.}When the base field is totally real and the algebra is ramified at every real place except one, there is an algorithm of Voight~\cite{vfuchsian} for computing the unit group of an order. In~\cite[p.~25]{vd}, Demb\'el\'e and Voight mention but do not describe an unpublished algorithm using this computation to speedup ideal principalization. This algorithm is provided in Magma\footnote{\texttt{IsPrincipal(<Any> I, <GrpPSL2> Gamma) -> BoolElt, AlgQuatElt}} and improves on the algorithm of~\cite{kv}. 
Let~$\fld$ be the real cubic field of discriminant~$3132=2^2\cdot 3^3\cdot 29$, which is generated by an element~$t$ such that~$t^3-15t+6=0$. Let~$\alg$ be the quaternion algebra over~$\fld$ generated by two elements~$i,j$ such that~$i^2=-1$,~$j^2=(141t^2+57t-2092)/2$ and~$ij=-ji$. The algebra~$\alg$ is ramified at two of the three real places of~$\fld$ and no finite place. %
We compute a maximal order in~$\alg$ and then construct the reduction structure in~$14$ seconds. 
We produce a random integral ideal of norm~$\prm$ for each prime~$\prm$ having norm less than~$100$ and compute a generator for each of them with our algorithm. The computation takes~$1.5$~seconds per ideal on average with a maximum of~$4.6$~seconds. With Magma we compute the unit group~$\order^\times$ in~$8$~minutes and then compute generators for the same ideals with the units-assisted algorithm~\cite[p.~25]{vd} provided by Magma. The computation takes~$1$ hour per ideal on average with a maximum of~$17$h, and~$10$ of the~$23$ ideals take less that~$0.5$~seconds. Magma tends to return smaller generators than our algorithm. Magma is fast whenever there exists a small generator and our algorithm is faster when this is not the case.

\begin{figure}[tbh]
\centering
\includegraphics*[height=9.5cm,keepaspectratio=true]{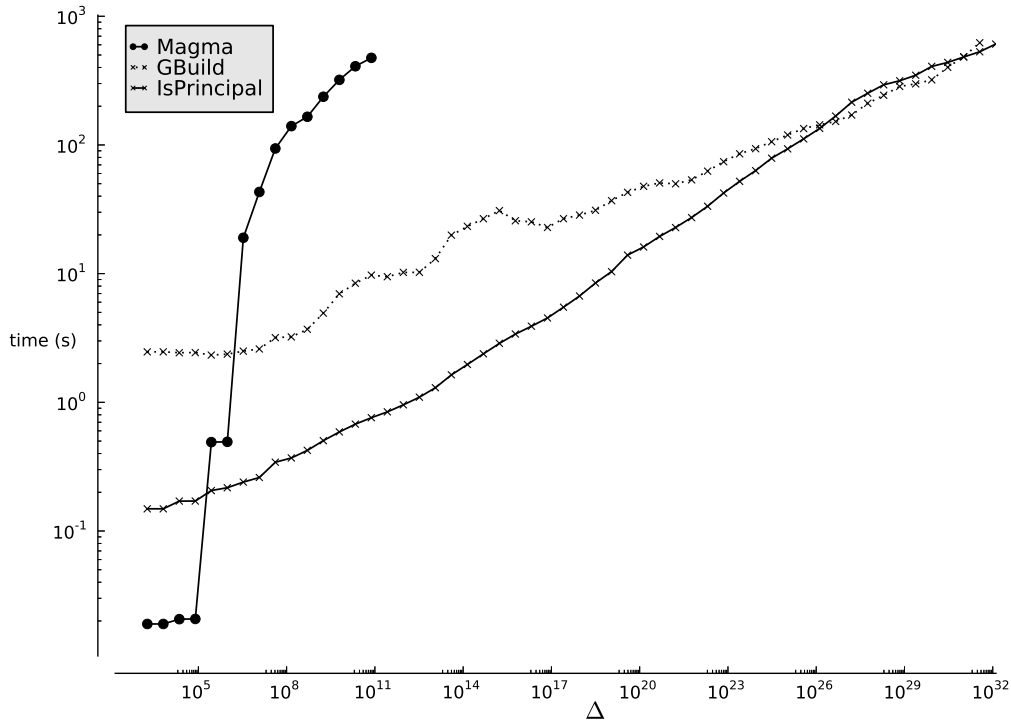}
\caption{Running time of the algorithms}\label{timings}
\end{figure}

\paragraph{Example 4.}In order to understand the practical behaviour of the algorithms, we conduct the following experiment. We draw algebras~$\alg$ and ideals~$I$ at random\footnote[2]{Let~$x$ be uniformly distributed in~$[0,70]$. This value controls the size of the discriminant. Let~$k$ be~$1$ or~$2$, each with probability~$1/2$. This is the number of prime factors of the discriminant of the algebra. Let~$t$ be uniformly distributed in~$[0,1]$. This value controls the part of the size of the discriminant coming from the base field or from the algebra. Let~$d$ be the smallest fundamental discriminant larger that~$\exp(tx/4)$, and let~$\fld=\Q(\sqrt{d})$.
Let~$\prm_1$ be the prime of~$\Z_\fld$ with smallest norm larger than~$\exp((1-t)x/2k)$, and if~$k=2$ let~$\prm_2$ be the prime with smallest norm larger than~$N(\prm_1)$. Let~$\alg$ be the quaternion algebra ramified exactly at~$\prm_i$ for~$i\le k$ and at~$k\,\mathrm{mod}\, 2$ real places of~$\fld$, and let~$\order$ be a maximal order in~$\alg$. Let~$\Delta = d^4N(\delta_\alg)^2$ be the absolute discriminant of~$\alg$. Let~$y$ be uniformly distributed in~$[0,1]$, and let~$\prm$ be the prime of~$\Z_\fld$ of smallest norm larger than~$y\Delta^{1/2}$, coprime to~$\delta_\alg$ and such that the class of~$\prm$ in~$\clalg$ is trivial.
Finally, let~$I$ be a random integral right $\order$-ideal of norm~$\prm$.}. In every random test case, we compute our reduction structure with Algorithm~\ref{gbuild} (\texttt{GBuild}), and we compute a generator of the ideal~$I$ with Algorithm~\ref{isprincipal} (\texttt{IsPrincipal}). We also compute a generator of~$I$ with the function provided by Magma. In every case, we interrupt any algorithm that takes more than~$1000$ seconds to terminate.
The result of~$15\,000$ such test cases is plotted in Figure~\ref{timings}: the discriminant~$\Delta$ and the time are both in logarithmic scale, and each plot~$(D,T)$ is such that~$T$ is the average of the running time of the algorithm over the discriminants~$\Delta\in[D/10,10D]$. We do not plot the running time when more than~$50\%$ of the executions were interrupted, since the corresponding value is no longer meaningful.


\begin{acknowledgements}
  I would like to thank Karim Belabas and Andreas Enge for helpful discussions and careful reading of early versions of this paper. I also want to thank an anonymous referee many comments and corrections and for suggesting a deterministic algorithm for computing a local generator and Pierre Lezowski for explaining to me Euclidean algorithms over matrix rings.
\end{acknowledgements}

\affiliationone{%
Aurel Page\\
Univ. Bordeaux, IMB, UMR 5251,\\
  F-33400 Talence, France.\\
  \email{aurel.page@math.u-bordeaux1.fr}}
\affiliationtwo{%
CNRS, IMB, UMR 5251,\\
  F-33400 Talence, France.}
\affiliationthree{%
INRIA,\\
  F-33400 Talence, France.}
%
\end{document}